\documentclass{amsart}
\usepackage{amsmath}
\usepackage{amsfonts}
\usepackage{amsthm}
\usepackage{amssymb}
\usepackage{enumerate}
\usepackage[all]{xy}
\usepackage{cite}
\usepackage{mathrsfs}
\usepackage{color}
\usepackage{xcolor}
\usepackage{hyperref}
\usepackage{fancyhdr}
\hypersetup{
	colorlinks=true,
	anchorcolor=blue,
	linkcolor=blue,
	filecolor=blue,
	urlcolor=blue,	
	citecolor=blue,
	bookmarks=true,
	bookmarksopen=true,
	pdfborder=000
}
\numberwithin{equation}{section}
\theoremstyle{plain}
\newtheorem{thm}{Theorem}[section]

\newtheorem{cor}[thm]{Corollary}

\newtheorem{lemma}[thm]{Lemma}
\newtheoremstyle{noparens}%
 {}{}%
 {\itshape}{}%
 {\bfseries}{.}%
 { }%
 {\thmname{#1}\thmnumber{ #2}\mdseries\thmnote{ #3}}
\theoremstyle{noparens}
\newtheorem{lemmaNoParens}[thm]{Lemma}
\newtheorem{thmNoParens}[thm]{Theorem}
\newtheorem{propositionNoParens}[thm]{Proposition}
\theoremstyle{definition}
\newtheorem{defn}[thm]{Definition}
\theoremstyle{remark}
\newtheorem{rmk}[thm]{Remark}

\makeatletter

\newcommand{\Rmnum}[1]{\expandafter\@slowromancap\romannumeral #1@}
\@namedef{subjclassname@2020}{\textup{2020} Mathematics Subject Classification}
\makeatother
\begin{document}
\title{The Kobayashi metric and Gromov hyperbolicity on pseudoconvex domains of finite type in $\mathbb{C}^2$}
\thanks{Haichou Li is supported by NSFC (No. 12226318, No. 12226334).}
\author{Haichou Li\textsuperscript{1} $\&$ Xingsi Pu\textsuperscript{2,3} $\&$ Lang Wang\textsuperscript{2,3}}
\address{$1.$ College of Mathematics and informatics, South China Agricultural University, Guangzhou, 510640, China}
\address{$2.$ HLM, Academy of Mathematics and Systems Science,
Chinese Academy of Sciences, Beijing, 100190, China}
\address{$3.$ School of
Mathematical Sciences, University of Chinese Academy of Sciences,
Beijing, 100049, China }
\email{hcl2016@scau.edu.cn,\:puxs@amss.ac.cn,\:wanglang2020@amss.ac.cn}

\begin{abstract}
In this paper, we obtain a more precise estimate of Catlin-type distance for smoothly bounded pseudoconvex domain of finite type in $\mathbb{C}^2$. As an application, we get an alternative proof of the Gromov hyperbolicity of this domain equipped with the Kobayashi distance.
\end{abstract}

\maketitle

\section{Introduction}
\indent{}In several complex variables, the Kobayashi metric is an important invariant metric under biholomorphic mappings. And the Gromov hyperbolicity of this metric has been studied in recent years. In\ \cite{2000Gromov}, Balogh and Bonk firstly obtained a sharp estimate of Kobayashi distance on bounded strongly pseudoconvex domain in $\mathbb{C}^n$. By using the estimate, they obtained the Gromov hyperbolicity of this domain. Later Zimmer \cite{zimmer2016gromov} proved the Gromov hyperbolicity of convex domain of finite type in $\mathbb{C}^n$ by scaling method. After his work, Fiacchi \cite{2022Gromov} proved the case of pseudoconvex domain of finite type in $\mathbb{C}^2$.

\indent{}On the other hand, Catlin \cite{1989Estimates} and McNeal \cite{mcneal1994estimates} constructed polydisks on pseudoconvex domain of finite type in $\mathbb{C}^2$ or convex domain of finite type in $\mathbb{C}^n$. Moreover, McNeal also defined a local pseudodistance near a boundary point. Recently, Wang \cite{wang22} used it to get an estimate of Kobayashi distance on convex domain of finite type in $\mathbb{C}^n$. And it can give a different proof of Zimmer's result in \cite{zimmer2016gromov}.\\
\indent After Catlin's work, Herbort \cite{Herbort2005} used the corresponding local pseudodistance to get an estimate of Kobayashi distance on pseudoconvex domain of finite type in $\mathbb{C}^2$. Inspired by Herbort and Wang, we improve the estimate in \cite[Theorem 1.12]{liu2023bi}, which can obtain the Gromov hyperbolicity of $(\Omega,d_{\widetilde{K}})$. Here $d_{\widetilde{K}}$ is the Catlin-type distance (see Section \ref{sec3.4}). Since Catlin-type metric is comparable to Kobayashi metric, it can give an alternative proof of Theorem 1.1 in \cite{2022Gromov} which is originally stated as follows.
\begin{thmNoParens}\label{finite}
Let $\Omega\subset\mathbb{C}^2$ be a bounded pseudoconvex domain of finite type with Kobayashi metric $d_K$, then $(\Omega,d_K)$ is a Gromov hyperbolic space.
\end{thmNoParens}

\indent{}Let $\Omega\subset\mathbb{C}^2$ be a smoothly bounded pseudoconvex domain of finite type and $U$ be a neighborhood of $\partial\Omega$. After shrinking $U$, Herbort\cite{Herbort2005} defined a pseudodistance $d(\cdot,\cdot)$ in $U$ (see Section \ref{sec3.2}). We extend it to $\Omega$ which is denoted by $D(\cdot,\cdot)$,  and set the function $g:\Omega\times\Omega\rightarrow \mathbb{R}$ by
\[
g(x,y)=2\log\dfrac{D(x,y)+\delta(x)\vee\delta(y)}{\sqrt{\delta(x)\delta(y)}}
\]
where $a\vee{b}:=\max\{a,b\}$. \\
\indent{}Note that the construction of $g$ was already present in \cite{2000Gromov}, and similar to \cite[Corollary 1.3]{2000Gromov} we have the following main result.
\begin{thm}\label{guji}
Let $\Omega\subset\mathbb{C}^2$ be a smoothly bounded pseudoconvex domain of finite type, then there is a constant $C>0$ such that
\[
g(x,y)-C\leq{}d_{\widetilde{K}}(x,y)\leq{}g(x,y)+C
\]
for any $x,y\in\Omega$. Here $d_{\widetilde{K}}$ is the Catlin-type distance.
\end{thm}

 As an application, we prove the following theorem which implies Theorem \ref{finite}.
\begin{thm}\label{thm1.3}
Let $\Omega\subset\mathbb{C}^2$ be a smoothly bounded pseudoconvex domain of finite type, then $(\Omega,d_{\widetilde{K}})$ is a Gromov hyperbolic space.
\end{thm}

From Theorem \ref{guji}, we know that $D(a,b)\asymp\operatorname{exp}(-(a|b)_w^{d_{\widetilde{K}}})$
for $a,b\in\partial\Omega$ with a fixed point $w\in\Omega$. Combining Propositon \ref{ext}, we have the following extension result for quasi-isometry.
\begin{cor}\label{cor4.5}
Suppose $\Omega_1,\Omega_2\subset\mathbb{C}^2$ are smoothly bounded pseudoconvex domains of finite type, and $f:\Omega_1\rightarrow\Omega_2$ is a quasi-isometry with respect to the Kobayashi distances. Then $f$ extends continuously to a map $\bar{f}:\overline{\Omega}_1\rightarrow\overline{\Omega}_2$ such that $\bar{f}|_{\partial\Omega_1}:(\partial\Omega_1,D_1)\rightarrow(\partial\Omega_2,D_2)$ is a power quasisymmtery with respect to the  pseudodistances. Moreover, it is bi-H$\ddot{o}$lder continuous with respect to the pseudodistances.
\end{cor}

\begin{rmk}
As $|a-b|^m\lesssim D(a,b)\lesssim|a-b|$ for $a,b\in\partial\Omega$, it implies \cite[Theorem 1.2]{liu2023bi} that $\bar{f}|_{\partial\Omega_1}$ is bi-H$\ddot{\operatorname{o}}$lder continuous with respect to the Euclidean distance.

\end{rmk}

The paper is organized as follows. In section \ref{sec2} we give the preliminaries will be used. In section \ref{sec3} we recall the construction of polydisks and Catlin-type\ metric. In section \ref{sec4} we give the proofs of Theorem \ref{guji} and Theorem \ref{thm1.3}.

\section{Preliminaries}\label{sec2}
\subsection{Notations.}(1) We denote $|z|$ as the standard Euclidean norm for $z\in\mathbb{C}^n$ and let $|z_1-z_2|$ denote the Euclidean distance of $z_1,z_2\in\mathbb{C}^n$.

\noindent{}(2) Suppose $\Omega$ is a bounded domain in $\mathbb{C}^n$, for $x\in\mathbb{C}^n$ we set
\[
\delta(x):=\inf\{|x-y|:y\in\partial\Omega\}
\]
 and denote $N_{\epsilon}(\partial\Omega):=\{x\in\mathbb{C}^n:\delta(x)<\epsilon\}$ whenever $\epsilon>0$.\\
\noindent{}(3) Let $\Delta\subset{\mathbb{C}}$ be the unit disk and $z,w\in\Delta.$ We denote $d_{\Delta}(z,w)$ as Poincare distance in $\Delta$.

\noindent{}(4) Let $f,g$ be functions, we denote $f \lesssim g$ if there is $C>0$ such that $f\leq Cg$, and we write $f\asymp g$ if $f\lesssim g$ and $g\lesssim f$.

\noindent (5) For $z\in\mathbb{C}^n$, we denote $z_i$ as the $i$-th component of $z$.

\subsection{Boundary projection.}
We give the relationship between a point and its boundary projection. And the proof can be found in \cite[Lemma 2.1]{2000Gromov}.

\begin{lemma}\label{lem2.1}
Suppose $\Omega=\{z\in\mathbb{R}^n:r(z)<0\}$ is a bounded domain with $C^2$-smooth boundary. Then there exists a constant $\delta_0>0$ such that

(1) For every point $x\in{N_{\delta_0}(\partial\Omega)}$, there is a unique point $\pi(x)\in\partial\Omega$ such that $|x-\pi(x)|=\delta(x)$.

(2) The signed distance function
\begin{align*}
\rho(x)=
\begin{cases}
 -\delta(x)\qquad &x\in\Omega\\
 \delta(x) \qquad &x\in\mathbb{R}^n\setminus\Omega
\end{cases}
 \end{align*}
is $C^2$-smooth in $N_{\delta_0}(\partial\Omega)$.

(3) For the fibers of map $\pi:N_{\delta_0}(\partial\Omega)\rightarrow\partial\Omega$ we have
\[
\pi^{-1}(p)=(p-\delta_0\vec{n}(p),p+\delta_0\vec{n}(p)),
\]
where $\vec{n}(p)$ is the outer unit normal vector of $\partial\Omega$ at $p\in\partial\Omega.$

(4) The gradient of the defining function $\rho$ satisfies
\[
\nabla{\rho}(x)=\vec{n}(\pi(x)).
\]

(5) The projection $\pi:N_{\delta_0}(\partial\Omega)\rightarrow\partial\Omega$ is $C^1$-smooth.
\end{lemma}

Now we recall that for any $p\in\partial\Omega,$ the real tangent space $T_p\partial\Omega$ is given by
\[
T_p\partial\Omega=\{X\in\mathbb{C}^n:\operatorname{Re}\langle\bar{\partial}r(p),X\rangle=0\},
\]
and its complex tangent space is
\[
H_p\partial\Omega=\{X\in\mathbb{C}^n:\langle\bar{\partial}r(p),X\rangle=0\},
\]
where $\bar{\partial}r(p)=(\dfrac{\partial{r}}{\partial{\bar{z}_1}}(p),\cdots,\dfrac{\partial{r}}{\partial{\bar{z}_n}}(p))$. Here $\langle\cdot,\cdot\rangle$ is the standard Hermitian product. Therefore, for any $0\neq{}X\in\mathbb{C}^n$ it has a unique orthogonal decomposition $X=X_N+X_H$ with $X_H\in{}H_p\partial\Omega$ and $X_N\in{}N_p\partial\Omega$. Here $N_p\partial\Omega$ is the complex one-dimensional subspace of $\mathbb{C}^n$ orthogonal to $H_p\partial\Omega$.

\subsection{The invariant metrics.} A $Finsler\ metric$ in $\Omega$ is an upper continuous map $F:\Omega\ \times\ \mathbb{C}^n\rightarrow[0,\infty)$ such that $F(z,tX)=|t|F(z,X)$ for all $z\in\Omega,\ t\in\mathbb{C},\ X\in\mathbb{C}^n.$ And the corresponding distance $d_F$ is defined by
\begin{align*}
d_F(x,y):=\inf\{L_F(\gamma):\gamma:[0,1]\rightarrow\Omega\ \operatorname{is\ a\ piecewise}\ C^1\operatorname{-smooth\ curve}\\
\operatorname{with}\ \gamma(0)=x,\gamma(1)=y\},
\end{align*}
where $L_F(\gamma):=\int_0^1F(\gamma(t),\dot{\gamma}(t))dt.$

For a domain $\Omega\subset\mathbb{C}^n$, the Kobayashi metric  $K(z,X)$ is defined by
\[
K(z,X):=\inf\{|\xi|:\exists{f\in \operatorname{Hol}(\Delta,\Omega),\ \operatorname{with}\ f(0)=z,d(f)_0(\xi)=X}\}
\]
for $z\in\Omega,\ 0\neq X\in\mathbb{C}^n$, and we denote $d_K$ the corresponding Kobayashi distance. Moreover, the Caratheodory distance is defined by
\[
d_C(p,q):=\sup\{d_{\Delta}(f(p),f(q)):f:\Omega\rightarrow \Delta\operatorname{ is\ holomorphic}\}.
\]

\noindent{}Due to the result in \cite[Theorem 1]{royden2006remarks}, we know that $d_K(p,q)\geq{}d_C(p,q)$ for any $p,q\in\Omega$.
\subsection{Gromov hyperbolicity.}
Now we introduce some definitions and properties about Gromov hyperbolicity. Refer to \cite{Metric99} for details.

Let $(X,d)$ be a metric space. The $Gromov\ product$ of two points $x,y\in{X}$ with respect to $w\in{X}$ is defined as follows
\[
(x|y)_w:=\dfrac{1}{2}(d(x,w)+d(y,w)-d(x,y)).
\]
\indent Recall that $(X,d)$ is a $geodesic$ metric space if  any $x,y\in{X}$ can be joined by a geodesic segment, i.e., the image of an isometry $\gamma:[0,d(x,y)]\rightarrow X$ with $\gamma(0)=x$ and $\gamma(d(x,y))=y$. And we call a metric space $(X,d)$ is $proper$ if any closed ball in $X$ is compact. We say that a proper geodesic metric space $(X,d)$ is $Gromov\ hyperbolic$ if there is a constant $\delta\geq0$ such that
\[
(x|y)_w\geq(x|z)_w\wedge(z|y)_w-\delta
\]
for any $x,y,z,w\in{X}$, where $a\wedge{b}:=\min\{a,b\}$.

\begin{rmk}\label{rmk}
The above condition can be rewritten as a 4-point condition:
\[
d(x,w)+d(y,z)\leq\max\{d(x,z)+d(y,w),d(x,y)+d(z,w)\}+2\delta
\]
for all $x,w,y,z\in X.$
\end{rmk}

For a Gromov hyperbolic space $X$, we recall the definition of Gromov boundary.
\begin{defn}
(1) A sequence $\left\{x_i\right\}$ in $X$ is called a $Gromov\text{ } sequence$ if $\left(x_i |x_j\right)_\omega \rightarrow \infty$ as $i$, $j \rightarrow \infty$.

(2) Two such sequences $\left\{x_i\right\}$ and $\left\{y_j\right\}$ are said to be $equivalent$ if $\left(x_i | y_i\right)_\omega \rightarrow \infty$ as $i \rightarrow \infty$.

(3) The $Gromov \text{ }boundary$ $\partial_G X$ of $X$ is defined to be the set of all equivalence classes of Gromov sequences, and $\overline{X}^G=X \cup \partial_G X$ is called the $Gromov \text{ }closure$ of $X$.

(4) For $a \in X$ and $b \in \partial_G X$, the Gromov product $(a | b)_\omega$ of $a$ and $b$ is defined by
$$
(a| b)_\omega=\inf \left\{\liminf _{i \rightarrow \infty}\left(a | b_i\right)_\omega:\left\{b_i\right\} \in b\right\} .
$$

(5) For $a, b \in \partial_G X$, the Gromov product $(a|b)_\omega$ of $a$ and $b$ is defined by
$$
(a|b)_\omega=\inf \left\{\liminf _{i \rightarrow \infty}\left(a_i | b_i\right)_\omega:\left\{a_i\right\} \in a \text { and }\left\{b_i\right\} \in b\right\} .
$$
\end{defn}

For a Gromov hyperbolic space $X$, one can define a class of visual metrics on $\partial_G X$ via the extended Gromov products. For any metric $\rho_G$ in this class, there exist a parameter $\epsilon>0$ and a base point $\omega\in X$ with
\begin{equation}\label{visual}
\rho_G(a, b) \asymp \exp \left(-\epsilon(a|b)_w\right), \quad \text { for } a, b \in \partial_G X .
\end{equation}

\begin{defn}
Suppose $f:X\rightarrow Y$ is a map between metric spaces, and $\lambda\geq1,k\geq0$ are constants. If for any $x,y\in X$
\[
\dfrac{1}{\lambda}d_X(x,y)-k\leq{}d_Y(f(x),f(y))\leq{}\lambda{}d_X(x,y)+k,
\]
we say that $f$ is an $(\lambda,k)$-$quasi$-$isometry$.
\end{defn}

\begin{lemmaNoParens}[{\cite[Part \Rmnum{3}: Theorem 1.9]{Metric99}}]\label{hyp}
Suppose $(X,d_X)$ and $(Y,d_Y)$ are geodesic metric spaces and let $f:X\rightarrow Y$ be a $(\lambda,k)$-quasi-isometric embedding, then $(X,d_X)$ is a Gromov hyperbolic space if $(Y,d_Y)$ is a Gromov hyperbolic space.
\end{lemmaNoParens}

\begin{defn}
Suppose $f:X\rightarrow Y$ is a bijection and $\lambda\geq1,\alpha>0$ are constants. If for distinct $x,y,z\in X$
\[
\frac{d_Y(f(x),f(y))}{d_Y(f(x),f(z))}\leq\eta_{\alpha,\lambda}\left(\frac{d_X(x,y)}{d_X(x,z)}\right),
\]
we say $f$ is an $(\alpha,\lambda)$-power quasisymmetry. Here
\begin{align*}
\eta_{\alpha,\lambda}(w)=
\begin{cases}
\lambda w^{1/\alpha}\qquad &0<w<1\\
\lambda w^{\alpha}\qquad &w\geq1.
\end{cases}
\end{align*}
\end{defn}

\begin{propositionNoParens}[{\cite[Theorem 6.5]{Bonk2000}}]\label{ext}
Suppose $f:X\rightarrow Y$ is a quasi-isometry between Gromov hyperbolic spaces. Then $f$ induces a power quasisymmetry $\bar{f}:\partial_GX\rightarrow\partial_GY$ with respect to the visual metrics. Moreover, $\bar{f}$ is a bi-H$\ddot{o}$lder map with respect to the visual metrics.
\end{propositionNoParens}

\section{Pseudoconvex domain of finite type in $\mathbb{C}^2$}\label{sec3}
In this section, we recall some properties of smoothly bounded pseudoconvex domain of finite type in $\mathbb{C}^2$ and it refers to \cite{1989Estimates,Herbort2005} for details. We always assume this domain is of finite type of at most $m$.

\subsection{Construction of polydisks.}\label{sec3.1}

We introduce a new coordinate near boundary points and construct the polydisks firstly.

\begin{lemmaNoParens}[{\cite[Lemma 2.1]{Herbort2005}}]\label{lm3.1}
 For any point $\zeta\in{U}$, where $U$ is a neighborhood of $\partial\Omega$, there is a radius $R>0$\ (independent on $\zeta$) and an injective holomorphic mapping $\Phi_{\zeta}:B(\zeta,R) \rightarrow \mathbb{C}^2$ such that \\
(1) $\Phi_{\zeta}$ is of form
\begin{align}\label{3.1}
\Phi_{\zeta}(z)=&(a_1(\zeta)(z_1-\zeta_1)+a_2(\zeta)(z_2-\zeta_2)\\
&+f_{m}(\zeta;\langle z-\zeta,L(\zeta)\rangle ), \langle z-\zeta,L(\zeta)\rangle )\nonumber
\end{align}
where $a_1(\zeta),a_2(\zeta)$ depend smoothly on $\zeta$ and $f_{m}$ is a holomorphic polynomial of degree $m$ with coefficients depend smoothly on $\zeta$ and $L(z)=(-\frac{\partial r}{\partial z_2}(z),\frac{\partial r}{\partial z_1}(z))$. Here $r$ is the defining function of $\Omega$.\\
(2) The defining function of $\Omega$ is normalized as $r=\rho_{\zeta}\circ\Phi_{\zeta}$ with
\[
\rho_{\zeta}(w)=r(\zeta)+\operatorname{Re}w_1+\sum_{k=2}^{m}P_{k}(\zeta;w_2)+O(|w_2|^{m+1}+|w||w_1|),
\]
where each $P_k(\zeta,\cdot)$ is a real-valued homogeneous polynomial that has no pure term.

\end{lemmaNoParens}

Now we define the polydisks. For a real-valued $k$-homogeneous polynomial
\[
P(z)=\sum_{j,l\geq{1},j+l=k}a_{j,l}z^j\bar{z}^l,
\]
we denote
\[
\parallel{p}\parallel:=\max\{|p(e^{i\theta})|:\theta\in\mathbb{R}\}.
\]
For any $\zeta\in{U}$ and a positive number $\delta>0$, we define
\[
\tau(\zeta,\delta):=\min\left\{\left(\frac{\delta}{\parallel{P_{l}(\zeta;\cdot)}\parallel}\right)^{1/l}:l=2,\cdots,m\right\}
\]
and the polydisk
\[
R_{\delta}(\zeta):=\{(z_1,z_2)\in\mathbb{C}^2:|z_1|<\delta,|z_2|<\tau(\zeta,\delta)\}.
\]
Denoting
\[
Q_{\delta}(\zeta):=\{z\in{}B(\zeta,R):\Phi_{\zeta}(z)\in{}R_{\delta}(\zeta)\},
\]
we have the following relations.

\begin{lemmaNoParens}[{\cite[Lemma 2.2]{Herbort2005}}]\label{lem3.2}
There exists $\widehat{C}>0,\hat{\delta}>0$ such that for all $0<\delta\leq\hat{\delta}$ and $\zeta_1,\zeta_2\in{U}$, if $\zeta_2\in{}Q_{\delta}(\zeta_1)$ we have \\
(1) $\zeta_1\in{}Q_{\widehat{C}\delta}(\zeta_2)$\\
(2) $Q_{\delta}(\zeta_1)\subset{Q_{\widehat{C}\delta}(\zeta_2)}$\\
(3) $\tau(\zeta_1,\delta)\leq{}\widehat{C}\tau(\zeta_2,\delta)\leq{}\widehat{C}^2\tau(\zeta_1,\delta)$\\
(4) $\tau(\zeta_1,c\delta)\leq c^{1/2}\tau(\zeta_1,\delta)$ for $c>1$, $\tau(\zeta_1,c\delta)\leq c^{1/m}\tau(\zeta_1,\delta)$ for $0<c\leq1$.

\end{lemmaNoParens}

\subsection{The pseudodistance.}\label{sec3.2}
Suppose $x,y\in{U}$, we define that
\[
M(x,y):=\{\delta:y\in{Q_{\delta}(x)}\},
\]
where $Q_{\delta}(x)$ is defined in Section \ref{sec3.1}. Set
\[
d'(x,y)=\inf M(x,y)
\]
if $M(x,y)$ is not empty (i.e., $|x-y|\leq R$) and $d'(x,y)=+\infty$ otherwise. For $x,y\in U$, denote
\begin{align*}
d(x,y):=\min\{d'(x,y),|x-y|\}.
\end{align*}
To extend it to $\Omega$, let
\begin{align*}
D(x,y):=
\begin{cases}
d(x,y)\qquad &x,y\in U\\
d_1(x,y)\qquad &otherwise
\end{cases}
\end{align*}
where
\begin{align*}
d_1(x,y)=
\begin{cases}
0,\qquad x=y\\
1,\qquad x\neq y.
\end{cases}
\end{align*}
Then the following lemma implies that $D(x,y)$ is a pseudodistance on $\Omega$.

\begin{lemmaNoParens}[{\cite[Lemma 3.1]{Herbort2005}}]\label{pseudo}
There is $C>0$ such that for any $x,y,z\in{\Omega}$, we have\\
(1) $D(x,y)=0$ if and only if $x=y$\\
(2) $D(y,x)\leq{}CD(x,y)$\\
(3) $D(x,y)\leq{}C(D(x,z)+D(z,y))$\\
(4) If $x\in U$, $D(x,\pi(x))\leq{}C\delta(x).$\\
(5) For $x,y\in U$ with $|x-y|<R$, then $d(x,y)\geq C^{-1}d'(x,y)$.
\end{lemmaNoParens}

Similar to \cite[Lemma 3.4]{wang22}, the inequality ($3$) above also has the following form.
\begin{lemma}\label{wang}
There exists a constant $\epsilon_0>0$ such that for any $\epsilon\leq\epsilon_0$ and $x,x_1,\cdots,x_n,y\in{\Omega}$, we have
\[
D^{\epsilon}(x,y)\leq{}2(D^{\epsilon}(x,x_1)+D^{\epsilon}(x_1,x_2)+\cdots+D^{\epsilon}(x_n,y)).
\]
\end{lemma}

\subsection{Local domains of comparison}
Suppose $\zeta\in{U}$ and
\[
\rho_{\zeta}(w)=r\circ\Phi^{-1}_{\zeta}(w)=r(\zeta)+\operatorname{Re}w_1+\sum_{k=2}^{m}P_{k}(\zeta;w_2)+O(|w_2|^{m+1}+|w||w_1|).
\]
For $t>0$, set
\[
J_{\zeta,t}(w):=\left(t^2+|w_1|^2+\sum_{k=2}^{m}\parallel{P_k(\zeta;\cdot)}\parallel^2|w_2|^{2k}\right)^{1/2}
\]
and
\[
U_{\zeta,t}:=\{w:|w|\leq{}R_0,|\rho_{\zeta}(w)|<sJ_{\zeta,t}(w)\}
\]
for a suitable $R_0<R$ and a small constant $s>0$. Then we have the following lemma.

\begin{lemmaNoParens}[{\cite[Lemma 4.1]{Herbort2005}}]
For sufficiently small $R_0<R$ and $s>0$ there exists a smooth real-valued function $H_{\zeta,t}$ in $U_{\zeta,t}\cup\{\rho_{\zeta}(w)<t\}$ such that\\
(1) For a small constant $\epsilon_0>0$, let $\rho_{\zeta,t}=\rho_{\zeta}+\epsilon_0H_{\zeta,t}$, then the domain
\[
\Omega_t^{\zeta}:=\{w:\rho_{\zeta}(w)<t\}\cup\{w\in{}U_{\zeta,t}:\rho_{\zeta,t}(w)<0\}
\]
is pseudoconvex.
\\
(2) There exist constants $r_0>0,s>0$ such that for any $\zeta\in\partial\Omega,\ t>0$ and $|w|<r_0$ the following holds
\[
P(w,t,s):=\left\{(z_1,z_2):|w_1-z_1|<sJ_{\zeta,t}(w),|w_2-z_2|<s\tau(\zeta,J_{\zeta,t}(w))\right\}\subset\Omega_t^{\zeta}
\]
provided that $\rho_{\zeta}(w)<0.$
\end{lemmaNoParens}

\subsection{Catlin-type metric.}\label{sec3.4}
Suppose $\Omega\ =\{z\in\mathbb{C}^2:r(z)<0\}$ is a bounded smooth pseudoconvex domain of finite type, and $\xi\in{\partial\Omega}$ is a point of type $m_{\xi}$. We may assume that $\frac{\partial{r}}{\partial{z_2}}(\xi)\neq 0,$ and define the following vector fields
\[
L_1=\dfrac{\partial}{\partial{z_1}}-(\dfrac{\partial{r}}{\partial{z_2}})^{-1}\dfrac{\partial{r}}{\partial{z_1}}\dfrac{\partial}{\partial{z_2}},\ L_2=\dfrac{\partial}{\partial{z_2}}
\]
in a neighborhood of $\xi$.

Note that $L_1r=0$ and $L_1,L_2$ form a basis of $T^{1,0}_z$ for all $z$ near $\xi$. For any $j,k>0$ set
\[
	\mathcal{L}_{j,k}(z):=\underbrace{L_1\cdots{}L_1}_{j-1}\underbrace{\bar{L}_1\cdots\bar{L}_1}_{k-1}\partial\bar{\partial}r(L_1,\bar{L}_1)(z)
\]
and
\[
C^{\xi}_l(z):=\max\{|\mathcal{L}_{j,k}(z)|:j+k=l\}.
\]
Let $X=b_1L_1+b_2L_2$ be a holomorphic tangent vector at $z$. The $Catlin\ metric$ is defined by
\begin{equation}
M_{\xi}(z,X):=\dfrac{|b_2|}{|r(z)|}+|b_1|\sum_{l=2}^{m_{\xi}}\left(\dfrac{C_l^{\xi}(z)}{|r(z)|}\right)^{\frac{1}{l}},
\end{equation}
which implies that

\begin{equation}\label{kob}
K(z,X)\leq{}C\left(\dfrac{|(\partial{}r(z),X)|}{\delta(z)}+\dfrac{|\langle{}L(z),X\rangle|}{\tau(z,\delta(z))}\right)
\end{equation}
holds for Kobayashi metric near $\xi$ from \cite[Theorem 1]{1989Estimates}.

Moreover, if we set
\begin{equation}
\widetilde{M}_{\xi}(z,X):=\dfrac{|X_N|}{\delta(z)}+|X_H|\sum_{l=2}^{m_{\xi}}\left(\dfrac{C_l^{\xi}(z)}{\delta(z)}\right)^{\frac{1}{l}},
\end{equation}
then $\widetilde{M}_{\xi}(z,X)\asymp{}K(z,X)$ in a neighborhood of $\xi$ by \cite[Lemma 2.2]{liu2023bi}.

By choosing an open neighborhood $U_i$ of $\xi_i$, which form a finite cover of $\partial\Omega$, then there exists a constant $\epsilon>0$ such that
\[
N_{\epsilon}(\partial\Omega)\subset{}\bigcup_{i=1}^sU_i.
\]
For $z\in\Omega\cap{}N_{\epsilon}(\partial\Omega)$, we denote $I_z=\{i:z\in{U_i}\}$ and set
\[
\widetilde{M}(z,X):=\max_{z\in{I_z}}\{\widetilde{M}_{\xi}(z,X)\}.
\]
Since it is upper semi-continuous, we can define a global Finsler metric in $\Omega$ by
\[
\widetilde{K}(z,X):=K(z,X)S(z,X)
\]
with a positive function $S(z,X)\asymp{1}$ and $\widetilde{K}(z,X)=\widetilde{M}(z,X)$ for $z\in{\Omega\cap{N_{\epsilon_0}}(\partial\Omega)}$. And it implies that $\widetilde{K}(z,X)\asymp{}K(z,X)$ in $\Omega$.  We may call it $Catlin$-$type\ metric$ and denote $d_{\widetilde{K}}$ the distance associated to this metric.

Note that the Catlin-type distance statisfies the following estimate.
\begin{lemmaNoParens}[{\cite[Lemma 2.3]{liu2023bi}}]\label{cat}
Let $\Omega\subset\mathbb{C}^2$ be a smoothly bounded pseudoconvex domain of finite type. Then there exists a constant $\epsilon>0$ such that, for any $x,y\in{}N_{\epsilon}(\partial\Omega)$, we have
\begin{equation}
d_{\widetilde{K}}(x,y)\geq\left|\log\left(\dfrac{\delta(y)}{\delta(x)}\right)\right|.
\end{equation}
Moreover, if $\pi(x)=\pi(y)$, there is
\begin{equation}
d_{\widetilde{K}}(x,y)=\left|\log\left(\dfrac{\delta(y)}{\delta(x)}\right)\right|.
\end{equation}
\end{lemmaNoParens}

\section{Estimate of Catlin-type distance}\label{sec4}
In this section, we prove Theorem \ref{guji} and Theorem \ref{thm1.3}. We firstly prove Theorem \ref{guji} and state it again for convenience.
\begin{thm}\label{thm4.1}
Let $\Omega\subset\mathbb{C}^2$ be a smoothly bounded pseudoconvex domain of finite type, then there is a constant $C>0$ such that
\begin{equation}\label{est}
g(x,y)-C\leq{}d_{\widetilde{K}}(x,y)\leq{}g(x,y)+C
\end{equation}
for any $x,y\in\Omega$.
\end{thm}

Similar to \cite[Lemma 4.4]{wang22}, we will prove the following lemma firstly.
\begin{lemma}\label{lower}
For any $c>0$, there is a constant $c_0>0$ such that for $x,y\in U\cap\Omega$, if $y\in\partial{Q_{c\delta(x)}(x)}$ then we have
\begin{equation}
d_{\widetilde{K}}(x,y)\geq{c_0.}
\end{equation}

\end{lemma}
\begin{proof}
Since Catlin-type metric is comparable to Kobayashi metric, we just need to prove the lemma for Kobayashi distance. If $c\geq4\widehat{C}$, where $\widehat{C}$ is the constant in Lemma \ref{lem3.2}, then by Lemma 6.1 in \cite{Herbort2005} we have
\[
y\notin Q_{4\widehat{C}\delta(x)}(x)\supset Q_{2\widehat{C}\delta(x)}(\pi(x)).
\]
And Lemma 5.3 in\cite{Herbort2005} shows that if $b\notin Q_{2\delta(a)}(\pi(a))$, then
\[
d_{K}(a,b)\geq c_0.
\]
It deduces that the lemma holds for $c\geq 4\widehat{C}.$

 If $c<4\widehat{C}$, then we have that
\[
\Phi_x(y)\in\{(z_1,z_2):|z_1|<4\widehat{C}\delta(x),|z_2|<\tau(x,4\widehat{C}\delta(x))\}\subset{P(w_x,t,s)}
\]
with a number $s>0$ which does not depend on $x,y$ and $t=4\widehat{C}\delta(x),w_x=\Phi_{\pi(x)}(x)$.\\
\indent We set
\[
f(z_1,z_2)=\dfrac{1}{C\delta(x)}(\Phi_x\circ\Phi^{-1}_{\pi(x)}(z_1,z_2))_1
\]
such that $\parallel{f}\parallel_{\infty}\leq{1}$ with suitable $C,$ which is independent of $x,y.$ To see this we observe that
\[
\Phi_x\circ\Phi^{-1}_{\pi(x)}(z_1,z_2)=(\Phi_x-\Phi_{\pi(x)})\circ\Phi^{-1}_{\pi(x)}(z_1,z_2)+(z_1,z_2),
\]
and hence for some constant $C_1>0$
\[
|f(z_1,z_2)|\leq{\dfrac{|(\Phi_x-\Phi_{\pi(x)})\circ\Phi^{-1}_{\pi(x)}(z_1,z_2)|+|z_1|}{C\delta(x)}}\leq{C_1+\dfrac{|z_1|}{C\delta(x)}}.
\]

Indeed, there exists a constant $C_2>0$ such that
\[
(\Phi_x-\Phi_{\pi(x)})\circ\Phi^{-1}_{\pi(x)}(z_1,z_2)\leq C_2|x-\pi(x)|=C_2\delta(x)
\]
from equality (\ref{3.1}) in Lemma \ref{lm3.1}. And on $P(w_x,t,s)$ the following holds
\[
|z_1|\leq{|z_1-(w_x)_1|+|(w_x)_1|\leq{sJ_{\pi(x),t}(w_x)+t\leq{C_3}\delta(x)}}
\]
with $f(w_x)=0$ and constant $C_3>0$. Moreover, Lemma 5.2 in\cite{Herbort2005} implies that there exists a function $\tilde{f}\in H^{\infty}(\Omega)$ such that $||\tilde{f}||_{\infty}\leq L^*||f||_{\infty}$ and $\tilde{f}(w_x)=f(w_x),\tilde{f}(w_y)=f(w_y)$. It means that
\begin{equation}\label{sj}
d_K(x,y)\geq{d_{C}(x,y)}\geq{d_{\Delta}(0,\dfrac{1}{L^*}\tilde{f}(w_y))}\geq{\dfrac{1}{2}\log\left(1+\dfrac{|\Phi_x(y)_1|^2}{(L^*C\delta(x))^2}\right)}.
\end{equation}
\indent Since $y\in\partial{}Q_{c\delta(x)}(x)$, then $|\Phi_x(y)_1|=c\delta(x)$ or $|\Phi_x(y)_2|=\tau(x,c\delta(x))$ holds. If $|\Phi_x(y)_1|=c\delta(x)$ holds, from (\ref{sj}) we obtain $d_{\widetilde{K}}(x,y)\geq c_0$. If $|\Phi_x(y)_2|=\tau(x,c\delta(x))$ holds, by choosing
\[
f(z_1,z_2)=\dfrac{1}{C\tau(x,\delta(x))}(\Phi_x\circ\Phi^{-1}_{\pi(x)}(z_1,z_2))_2
\]
with suitable $C$ we get that
\begin{equation}
d_K(x,y)\geq{d_{C}(x,y)}\geq{d_{\Delta}(0,\dfrac{1}{L^*}f(w_y))}\geq{\dfrac{1}{2}\log\left(1+\dfrac{|\Phi_x(y)_2|^2}{(L^*C\tau(x,\delta(x)))^2}\right)}.
\end{equation}
It implies $d_{\widetilde{K}}(x,y)\geq{}c_0$ whenever $y\in\partial{}Q_{c\delta(x)}(x)$.
\end{proof}

Next we prove the following result which gives the upper bound $d_{\widetilde{K}}(x,y)$ when $\delta(x)=\delta(y)$ and $D(x,y)\leq{c\delta(x)}$ with a constant $c>0$.
\begin{lemma}\label{upp}\label{denggao}
After shrinking $U$, then for any $c>0$ there exist a constant $C>0$ such that, for any $x,y\in{\Omega\cap{}U}$ with $\delta(x)=\delta(y)$ and $D(x,y)\leq{c\delta(x)}$, the following holds
\begin{equation}
d_{\widetilde{K}}(x,y)\leq{C}.
\end{equation}

\end{lemma}
\begin{proof}
We just need to prove the lemma for Kobayashi distance $d_K$. If $D(x,y)=|x-y|$, then Theorem 3.5 in \cite{liu2023bi} implies that there exists a constant $C_1>0$ such that
\[
d_{\widetilde{K}}(x,y)\leq2\log\frac{|x-y|+\delta(x)\vee\delta(y)}{\sqrt{\delta(x)\delta(y)}}+C_1.
\]
And it implies the lemma.\\
\indent If $D(x,y)=d'(x,y)$. From \cite[Main Theorem 2.1]{Herbort2005}, we have
\[
d_{K}(x,y)\leq C_2(\rho(x,y)+\rho(y,x))
\]
for some constant $C_2>0$ with
\[
\rho(x,y)=\log(1+\frac{d'(x,y)}{\delta(x)}+\frac{|\Phi_x(y)_2|}{\tau(x,\delta(x))}).
\]
Since $d'(x,y)\leq c\delta(x)$ and according to (4) in Lemma \ref{lem3.2}, there exists a constant $C_3>0$ such that $|\Phi_x(y)_2|\leq\tau(x,D(x,y))\leq\tau(x,c\delta(x))\leq C_3\tau(x,\delta(x))$. Consequently, we can conclude that $\rho(x,y)\leq C_4,\rho(y,x)\leq C_4$ with some $C_4>0$. It completes the proof.
\end{proof}

\noindent{}$Proof\ of\ Theorem\ \ref{thm4.1}.$ Now we proceed to prove Theorem \ref{thm4.1}, which is similar to \cite[Theorem 1.1]{2000Gromov}. Let $N:=N_{\epsilon_0}(\partial\Omega)\cap\Omega$ be a neighborhood such that Lemma \ref{lem2.1}, Lemma \ref{lower} and Lemma \ref{denggao} hold. Denote $K$ the closure of $\Omega\setminus{N}$, then $K$ is compact in $\Omega$. In the following dicussion, we will denote by $C$ positive constants only depending on $\epsilon_0$ and the various other constants that are associated with $\Omega$. It is important to note that the specific value of $C$ is not relevant and may change even within the same line.

(1) If $x,y\in K$, then
\[
0\leq{g(x,y)}\leq C,\ 0\leq{}d_{\widetilde{K}}(x,y)\leq C.
\]
Hence inequality (\ref{est}) is true.

(2) If $x\in K,\ y\in N$ or $y\in K,\ x\in N$. We may assume that $x\in K,\ y\in N$. Then
\[
\log{\frac{1}{\delta(y)}}-C\leq g(x,y)\leq\log{\frac{1}{\delta(y)}}+C.
\]

To get the upper bound of $d_{\widetilde{K}}$, we set $y'=\pi(y)-\epsilon_0\vec{n}(\pi(y))$. We know that $y'\in K,\ \pi(y')=\pi(y)$. From Lemma \ref{cat}, we know that
\[
d_{\widetilde{K}}(y,y')=\log{\frac{1}{\delta(y)}}+C.
\]
Since $x,y'\in K$, then $d_{\widetilde{K}}(x,y')\leq C$. It implies that $d_{\widetilde{K}}(x,y)\leq d_{\widetilde{K}}(x,y')+d_{\widetilde{K}}(y',y)\leq\log{\frac{1}{\delta(y)}}+C$. Hence we get the right half of inequality (\ref{est}).

To obtain lower bound for $d_{\widetilde{K}}(x,y)$. Let $\gamma$ be an arbitrary piecewise $C^1$-smooth curve in $\Omega$ joining $y$ and $x$. Then there is a first point $x'$ on the curve with $x'\in K$, which means that $\delta(x')=\epsilon_0$. Suppose $\alpha$ is the subcurve of $\gamma$ with endpoints $y$ and $x'$ and let $L_{\widetilde{K}}(\gamma)$ denote the length of the curve $\gamma$ with respect to the Finlser metric $\widetilde{K}$, then we have by Lemma \ref{cat}
\[
L_{\widetilde{K}}(\gamma)\geq L_{\widetilde{K}}(\alpha)\geq\log\frac{1}{\delta(y)}-C.
\]
Taking the infimum for $\gamma$, we obtain
\[
d_{\widetilde{K}}(x,y)\geq\log\frac{1}{\delta(y)}-C.
\]
It completes the proof.

(3) Now we may assume that $x,y\in N$. First we estimate the upper bound of $d_{\widetilde{K}}.$\\
Case 1. $D(x,y)\leq{\delta(x)\vee{\delta(y)}}$.\\
We may assume that $\delta(y)\geq\delta(x).$ Then
\[
\log\dfrac{\delta(y)}{\delta(x)}\leq g(x,y)\leq\log\dfrac{\delta(y)}{\delta(x)}+C.
\]
Let $x'=\pi(x)-\delta(y)\cdot\vec{n}(\pi(x))$, by Lemma \ref{cat} we get
\[
d_{\widetilde{K}}(x,x')={\log\dfrac{\delta(y)}{\delta(x)}}.
\]
From Lemma \ref{denggao} we obtain $d_{\widetilde{K}}(x',y)\leq{C}$, which means that
\[
d_{\widetilde{K}}(x,y)\leq d_{\widetilde{K}}(x,x')+d_{\widetilde{K}}(x',y)\leq{g(x,y)+C}.
\]

\noindent{}Case 2. $D(x,y)>{\delta(x)\vee{\delta(y)}}.$\\
We have that
\[
2\log\dfrac{D(x,y)}{\sqrt{\delta(x)\delta(y)}}\leq g(x,y)\leq2\log\dfrac{D(x,y)}{\sqrt{\delta(x)\delta(y)}}+C.
\]
Let $x'=\pi(x)-h_0\vec{n}(\pi(x))$ and $y'=\pi(y)-h_0\vec{n}(\pi(y))$ with $h_0=D(x,y)\wedge\epsilon_0$. Note that $\delta(x')=\delta(y')=h_0\leq D(x,y)$. Then the following holds
\[
d_{\widetilde{K}}(x',x)+d_{\widetilde{K}}(y',y)\leq{2\log\dfrac{D(x,y)}{\sqrt{\delta(x)\delta(y)}}}.
\]
Therefore, we only need to estimate the upper bound of $d_{\widetilde{K}}(x',y')$. If $h_0=\epsilon_0$, then both $x'$ and $y'$ lie in $K$, hence $d_{\widetilde{K}}(x',y')\leq C$. We may assume that $h_0=D(x,y)$. From Lemma \ref{pseudo} we get that $D(x',x)\leq{CD(x,y)}$ and $D(y',y)\leq{CD(x,y)}$, which means $D(x',y')\leq CD(x,y)$. By Lemma \ref{denggao}, we have
\[
d_{\widetilde{K}}(x',y')\leq{C}.
\]
It deduces that
 \[
 d_{\widetilde{K}}(x,y)\leq{g(x,y)}+C.
 \]

Now we estimate the lower bound of $d_{\widetilde{K}}$. We may assume that $\delta(y)\geq\delta(x)$.

\noindent{}Case 1. $D(x,y)\leq{\delta(x)\vee{\delta(y)}}$.\\
We have that
\[
\log\dfrac{\delta(y)}{\delta(x)}\leq g(x,y)\leq\log\dfrac{\delta(y)}{\delta(x)}+C
\]
and by Lemma \ref{cat}
\[
d_{\widetilde{K}}(x,y)\geq{\log\dfrac{\delta(y)}{\delta(x)}}-C,
\]
which is the desired inequality.\\
\noindent{}Case 2. $D(x,y)>\delta(x)\vee{\delta(y)}$.\\
We know that
\[
2\log\dfrac{D(x,y)}{\sqrt{\delta(x)\delta(y)}}\leq g(x,y)\leq2\log\dfrac{D(x,y)}{\sqrt{\delta(x)\delta(y)}}+C.
\]

 Let $\gamma:[0,1]\rightarrow\Omega$ be an arbitrary curve joining $x$ and $y$. Define $H=\underset{z\in\gamma}{\max}\ \delta(z)$. There exists a constant $t_0\in[0,1]$ such that $\delta(\gamma(t_0))=H.$ We consider two curves $\gamma_1=\gamma|_{[0,t_0]}$ and $\gamma_2=\gamma|_{[t_0,1]}$. There are two possibilities.\\
 (a) If $H\geq{h_0}$ with $h_0=D(x,y)\wedge\epsilon_0$, then
 \[
 L_{\widetilde{K}}(\gamma_1)\geq{\log\dfrac{h_0}{\delta(x)}},\ L_{\widetilde{K}}(\gamma_2)\geq{\log\dfrac{h_0}{\delta(y)}}.
\]
Thus we get that
 \[
 L_{\widetilde{K}}(\gamma)\geq{\log\dfrac{h_0^2}{\delta(x)\delta(y)}}\geq{2\log\dfrac{D(x,y)}{\sqrt{\delta(x)\delta(y)}}}-C
 \]
with $h_0\geq CD(x,y)$.

\noindent{}(b) If $H<h_0$.\\
\noindent Since $\delta(x)\leq{H},$ there exists a integer $k\in{\mathbb{N}}$ such that
\[
\dfrac{H}{2^k}<\delta(x)\leq{\dfrac{H}{2^{k-1}}}.
\]
We consider the curve $\gamma_1$ and define $0=s_0\leq{s_1<\cdots<s_k\leq{t_0}}$ as follows. Let
\[
s_j=\max\left\{s\in[0,t_0]:\delta(\gamma(s))=\dfrac{H}{2^{k-j}}\right\}
\]
for $j=1,2\cdots,k.$ Setting $z_j=\gamma(s_j)$, we note that
\[
1\leq\dfrac{\delta(z_j)}{\delta(z_{j-1})}\leq2
\]
for $j=1,\cdots,k.$\\
Let $\epsilon=\epsilon_0$ in Lemma \ref{wang} and denote $c_1=\left(\dfrac{1-1/2^{\epsilon}}{16}\right)^{\frac{1}{\epsilon}}$. There are two cases:

$(\romannumeral1)$. There is an index $l\in\{1,\cdots,k\}$ such that
\[
D(z_{l-1},z_l)>c_1\dfrac{D(x,y)}{2^{k-l}}.
\]
Note that for $t\in[s_{l-1},s_l]$ one have
\[
\delta(\gamma(t))\leq{\dfrac{H}{2^{k-l}}}.
\]
Denote $p=z_{l-1}$ and $q=z_l$. Since
\begin{equation}\label{equa}
D(p,q)>c_1\dfrac{D(x,y)}{2^{k-l}}\geq{c_1\delta(p)},
\end{equation}
we have that
 \[
 q\notin{Q_{c_1\delta(p)}(p)}.
 \]
It means $\gamma\cap\partial{Q_{c_1\delta(p)}(p)}\neq\emptyset.$ Hence the following set
\[
S^{p,q}=\{m\in\mathbb{N}:\exists u_1,\cdots,u_m\ \operatorname{so\ that}\ s_{l-1}=u_0<u_1<\cdots<u_m\leq{s_l},\]
\[
\gamma(u_{j})\in\partial{Q_{c_1\delta(p_{j-1})}(p_{j-1})},1\leq{j}\leq{}m\}
\]
is not empty with $p_{j-1}=\gamma(u_{j-1})$.

By Lemma \ref{lower} there exists $c_0>0$ such that
\[
L_{\widetilde{K}}(\gamma|_{[s_{l-1},s_l]})\geq\sum_{j=1}^md_{\widetilde{K}}(p_{j-1},p_j)\geq{c_0m},
\]
we know $S^{p,q}$ is a finite set. Denote $m_0=\max{}S^{p,q}$, then
\begin{align*}
D^{\epsilon}(p,q)&\leq{2(\sum_{j=1}^{m_0}D^{\epsilon}}(p_{j-1},p_j)+D^{\epsilon}(p_{m_0},q))\\
&\leq{2c_1^{\epsilon}(m_0+1)\left(\dfrac{H}{2^{k-l}}\right)^{\epsilon}}.
\end{align*}
Hence one obtain that
\[
m_0\geq{}\dfrac{1}{2}\left(\dfrac{D(x,y)}{H}\right)^{\epsilon}-1,
\]
which means
\[
L_{\widetilde{K}}(\gamma|_{[s_{l-1},s_l]})\geq{\dfrac{c_0}{2}}\left(\dfrac{D(x,y)}{H}\right)^{\epsilon}-c_0.
\]
Let $t_1=s_k\leq{t_0}$, then
\begin{align*}
L_{\widetilde{K}}(\gamma|_{[0,t_1]})&=L_{\widetilde{K}}(\gamma|_{[0,s_{l-1}]})+L_{\widetilde{K}}(\gamma|_{[s_{l-1},s_l]})+L_{\widetilde{K}}(\gamma|_{[s_l,s_k]})\\
&\geq{\log\left(\dfrac{\delta(z_{l-1})}{\delta(z_0)}\right)+\dfrac{c_0}{2}\left(\dfrac{D(x,y)}{H}\right)^{\epsilon}+\log\left(\dfrac{\delta(z_k)}{\delta(z_l)}\right)-C}\\
&\geq{}\log\left(\dfrac{H}{\delta(x)}\right)+\dfrac{c_0}{2}\left(\dfrac{D(x,y)}{H}\right)^{\epsilon}-C.
\end{align*}

($\romannumeral2$). If the following holds
\[
D(z_{j-1},z_j)\leq{c_1\dfrac{D(x,y)}{2^{k-j}}}
\]
for all $j=1,\cdots,k$, then
\begin{align*}
D^{\epsilon}(x,\gamma(t_1))\leq{2\sum_{j=1}^kD^{\epsilon}(z_{j-1},z_j)}\leq\frac{2c_1^{\epsilon}}{1-1/{2^{\epsilon}}}D^{\epsilon}(x,y)\leq{\dfrac{1}{8}D^{\epsilon}(x,y)}.
\end{align*}
Moreover, we also have
\[
L_{\widetilde{K}}(\gamma|_{[0,t_1]})\geq{}\log\dfrac{H}{\delta(x)}.
\]
Thus there are two cases:
\begin{align}
L_{\widetilde{K}}(\gamma|_{[0,t_1]})\geq{}\log\left(\dfrac{H}{\delta(x)}\right)+C\left(\dfrac{D(x,y)}{H}\right)^{\epsilon}-C
\end{align}
or

\begin{align}\label{case1}
L_{\widetilde{K}}(\gamma|_{[0,t_1]})\geq{}\log\dfrac{H}{\delta(x)}\ \operatorname{and}\ D^{\epsilon}(x,\gamma(t_1))\leq\dfrac{1}{8}D^{\epsilon}(x,y).
\end{align}

Applying same consideration to $\gamma_2$ we can find $t_2\in[t_0,1]$ such that one of the following alternatives holds
\begin{align}
L_{\widetilde{K}}(\gamma|_{[t_2,1]})\geq{}\log\left(\dfrac{H}{\delta(y)}\right)+C\left(\dfrac{D(x,y)}{H}\right)^{\epsilon}-C
\end{align}
or
\begin{align}\label{case2}
L_{\widetilde{K}}(\gamma|_{[t_2,1]})\geq{}\log\dfrac{H}{\delta(y)}\ \operatorname{and}\ D^{\epsilon}(y,\gamma(t_2))\leq\dfrac{1}{8}D^{\epsilon}(x,y).
\end{align}

Suppose (\ref{case1}) and (\ref{case2}) hold, we have
\begin{align*}
D^{\epsilon}(\gamma(t_1),\gamma(t_2))&\geq\dfrac{1}{2}D^{\epsilon}(x,y)-D^{\epsilon}(x,\gamma(t_1))-D^{\epsilon}(\gamma(t_2),y)
\\
&\geq{\dfrac{1}{4}D^{\epsilon}(x,y)\geq{\dfrac{1}{4}}H^{\epsilon}}
\end{align*}
which means that
\[
D(\gamma(t_1),\gamma(t_2))\geq\dfrac{D(x,y)}{4^{\frac{1}{\epsilon}}}\geq{\dfrac{H}{4^{\frac{1}{\epsilon}}}}=\dfrac{\delta(t_1)}{4^{\frac{1}{\epsilon}}}.
\]
By using the same discussion for (\ref{equa}) in $(\romannumeral1)$, we deduce that
\[
L_{\widetilde{K}}(\gamma|_{[t_1,t_2]})\geq{C\left(\dfrac{D(x,y)}{H}\right)^{\epsilon}}-C,
\]
which implies that
\begin{align*}
L_{\widetilde{K}}(\gamma)&=L_{\widetilde{K}}(\gamma|_{[0,t_1]})+L_{\widetilde{K}}(\gamma|_{[t_1,t_2]})+L_{\widetilde{K}}(\gamma|_{[t_2,1]})\\
&\geq2\log\dfrac{H}{\sqrt{\delta(x)\delta(y)}}+C\left(\dfrac{D(x,y)}{H}\right)^{\epsilon}-C
\end{align*}

\noindent{}and the last inequality holds for other cases. Set
\[
h(t)=2\log\dfrac{t}{\sqrt{\delta(x)\delta(y)}}+C\left(\dfrac{D(x,y)}{t}\right)^{\epsilon}
\]
we know that $h$ has a minimum if
\[
t=(\frac{C\epsilon}{2})^{1/\epsilon}D(x,y).
\]
And it implies that
\[
L_{\widetilde{K}}(\gamma)\geq{2\log\dfrac{D(x,y)}{\sqrt{\delta(x)\delta(y)}}-C}.
\]
If we take infimum for $\gamma$, then we get that
\[
d_{\widetilde{K}}(x,y)\geq{2\log\dfrac{D(x,y)}{\sqrt{\delta(x)\delta(y)}}}-C
\]
and the proof is completed.$\hfill\qed$

\bigskip
\noindent $Proof\ of\ Theorem\ \ref{thm1.3}.$ The following proof is similar to \cite[Theorem 1.4]{2000Gromov} and we present it here for the sake of completeness.

Suppose $r_{ij}\geq0$ are given numbers such that $r_{ij}\leq Cr_{ji}$ and $r_{ij}\leq  C(r_{ik}+r_{kj})$ for $i,j,k\in\{1,2,3,4\}$ with a constant $C\geq1$. Then we have $r_{12}r_{34}\leq 4 C^4((r_{13}r_{24})\vee(r_{14}r_{23}))$. Indeed, we assume that $r_{13}$ is the smallest of the quantities $r_{ij}$ appearing on the right hand side of this inequality. Then we know that
\[
r_{12}\leq  C(r_{13}+r_{32})\leq  C(C+1)r_{23}\leq 2C^2r_{23}
\]
and $r_{34}\leq  C(r_{31}+r_{14})\leq 2 C^2r_{14}$. It deduces the inequality.

Set $x_i\in\Omega\ (i=1,2,3,4)$ and $r_{ij}=d_{ij}+h_i\vee h_j$ where $d_{ij}=D(x_i,x_j),h_i=\delta(x_i)$, then we have $r_{ij}\leq C_1r_{ji}$ and $r_{ij}\leq C_1(r_{ik}+r_{kj})$ for some constant $C_1\geq1$. It implies that
\[
r_{12}r_{34}\leq 4C_1^2((r_{13}r_{24})\vee(r_{14}r_{23}))
\]
and
\begin{align*}
&(d_{12}+h_1\vee h_2)(d_{34}+h_3\vee h_4)\\
&\leq 4C_1^2(((d_{13}+h_1\vee h_3)(d_{24}+h_2\vee h_4))\vee((d_{14}+h_1\vee h_4)(d_{23}+h_2\vee h_3))).
\end{align*}
From inequality (\ref{est}), we obtain
\[
d_{\widetilde{K}}(x_1,x_2)+d_{\widetilde{K}}(x_3,x_4)\leq (d_{\widetilde{K}}(x_1,x_3)+d_{\widetilde{K}}(x_2,x_4))\vee(d_{\widetilde{K}}(x_1,x_4)+d_{\widetilde{K}}(x_2,x_3))+C',
\]
for some constant $C'>0$ independent of points. By Remark \ref{rmk}, this inequality is equivalent to the Gromov hyperbolicity of the space $(\Omega,d_{\widetilde{K}})$.
$\hfill\qed$

\vspace{0.3cm} \noindent{\bf Acknowledgements}. The authors would like to thank Professor Jinsong Liu for many precious suggestions. We would also like to thank the referee for a careful reading and valuable comments.

\bibliography{reference}
\bibliographystyle{plain}{}
\end{document}